\documentclass[preprint,review,10pt]{amsart}
\usepackage{amsmath,amssymb,amsfonts,xspace}
\newtheorem{theorem}{Theorem}[section]

\theoremstyle{definition}
\newtheorem{definition}[theorem]{Definition}
\newtheorem{example}[theorem]{Example}

\newtheorem{proposition}[theorem]{Proposition}
\newtheorem{corollary}[theorem]{Corollary}

\theoremstyle{remark}
\newtheorem{remark}[theorem]{Remark}

\numberwithin{equation}{section}

\begin{document}

\title{ (weakly) $(\alpha,\beta)$-prime hyperideals in commutative multiplicative hyperrings }

\author{Mahdi Anbarloei}
\address{Department of Mathematics, Faculty of Sciences,
Imam Khomeini International University, Qazvin, Iran.
}

\email{m.anbarloei@sci.ikiu.ac.ir}


\subjclass[2010]{ Primary 20N20; Secondary 16Y99}


\keywords{  $(\alpha,\beta)$-prime, weakly $(\alpha,\beta)$-prime, $(\alpha,\beta)$-zero.}

\begin{abstract}
 Let $H$ be a commutative  multiplicative hyperring and $\alpha,\beta \in \mathbb{Z}^+$. A proper hyperideal $P$ of $H$ is called (weakly) $(\alpha,\beta)$-prime  if ($0 \notin x^{\alpha} \circ y \subseteq P$ ) $x^{\alpha} \circ y \subseteq P$ for $x,y \in H$ implies $x^{\beta} \subseteq P$ or $y \in P$.    In this paper, we aim to  investigate  (weakly) $(\alpha,\beta)$-prime hyperideals and then we  present some properties of them. 
\end{abstract}
\maketitle

\section{Introduction}
Delving into the study of prime ideal generalizations has emerged as a profoundly intriguing and groundbreaking pursuit in the realm of commutative ring theory.
In a recent study \cite{Khashan1},   Khashan and  Celikel  presented  $(\alpha,\beta)$-prime ideals which is an  intermediate class between  prime ideals and $(\alpha,\beta)$-closed ideals. 
 Let   $\alpha, \beta \in \mathbb{Z}^+$. A proper ideal $I$ of a commutative ring $R$ refers to an $(\alpha,\beta)$-prime ideal if for $x,y \in R$, $x^{\alpha}y \in I$ implies either $x^{\beta} \in I$ or $y \in I$. Moreover, the authors  generalized this notion to weakly $(\alpha,\beta)$-prime ideals  in \cite{Khashan2}. A proper ideal $I$ of $R$ is called a weakly $(\alpha,\beta)$-prime ideal if for $x,y \in R$ with $0 \neq x^{\alpha}y \in I$, then either $x^{\beta} \in I$ or $y \in I$.

Several key concepts in modern algebra were expanded by extending their underlying structures to hyperstructures. In 1934 \cite{marty}, the French mathematician F. Marty pioneered the notion of hyperstructures or multioperations, where an operation yields a set of values rather than a single value.
 Subsequently, numerous authors have contributed to the advancement of this novel area of modern algebra \cite{f1,f2,f3,f4,f5,f7,f8,f9}. 
An importan type of the algebraic hyperstructures  called the multiplicative hyperring was introduced by Rota in 1982 \cite{f14}. In this hyperstructure,  the multiplication is a hyperoperation and  the addition is an operation.  Multiplicative hyperrings are richly demonstrated and characterized in  \cite{ameri5, ameri6,  anb4,   Kamali, Ghiasvand, Ghiasvand2, f16, ul}.  Dasgupta studied the prime and primary hyperideals in multiplicative hyperrings in \cite{das}. The idea of $(\alpha,\beta)$-closed hyperideals in a multiplicative hyperring was proposed in \cite{anb7}. A proper hyperideal $P$ of a multiplicative hyperring $H$ is said to be $(\alpha,\beta)$-closed if for $x \in H$ with $x^{\alpha} \subseteq P$, then $x^{\beta} \subseteq P$. Motivated from this notion, the aim of this research work is to  introduce and study   the notion of $(\alpha,\beta)$-prime  hyperideals in  a commutative multiplicative hyperring.   Several specific results  are given to illustrate  the structures of the new notion. Furthermore,  we  extend this notion to weakly $(\alpha,\beta)$-prime  hyperideals. 
We present some characterizations of weakly $(\alpha,\beta)$-prime  hyperideals on cartesian product of commutative multiplicative hyperring.  

\section{Some basic definitions concerning multiplicative hyperrings}
In this section we give some basic definitions and results  which we need to
develop our paper.  \cite{f10} A hyperoperation $``\circ" $ on nonempty set $I$ is a mapping from $I \times I$ into $P^*(I)$ where $P^*(I)$ is the family of all nonempty subsets of $I$. If $``\circ" $ is a hyperoperation on $I$, then $(I,\circ)$ is called hypergroupoid. The hyperoperation on $I$ can be extended to  subsets of $I$. Let $I_1,I_2$ be two subsets of $I$ and $a \in I$, then $I_1 \circ I_2 =\cup_{a_1 \in I_1, a_2 \in I_2}a_1 \circ a_2,$ and $ I_1 \circ a=I_1 \circ \{a\}.$ A hypergroupoid $(I, \circ)$ is called  a semihypergroup if $\cup_{v \in b \circ c}a \circ v=\cup_{u \in a \circ b} u \circ c $ for all $a,b,c \in I$ which means $\circ$ is associative. A semihypergroup $I$ is said to be a hypergroup if  $a \circ I=I=I\circ a$ for all  $a \in I$. A nonempty subset $J$ of a semihypergroup $(I,\circ)$ is called a
subhypergroup if  we have $a \circ J=J=J \circ a$ for all $a \in J$. 

\begin{definition}
 \cite{f10} An algebraic structure $(H,+,\circ)$ is said to be commutative multiplicative hyperring if 
 \begin{itemize}
\item[\rm(1)]~ $(H,+)$ is a commutative group; 
\item[\rm(2)]~ $(H,\circ)$ is a semihypergroup; 
\item[\rm(3)]~ for all $a, b, c \in H$, we have $a\circ (b+c) \subseteq a\circ b+a\circ c$ and $(b+c)\circ a \subseteq b\circ a+c\circ a$; 
\item[\rm(4)]~ for all $a, b \in H$, we have $a\circ (-b) = (-a)\circ b = -(a\circ b)$; 
\item[\rm(5)]~ for all $a,b \in H$, $a \circ b =b \circ a$.
 \end{itemize}
 \end{definition}
If in (3) the equality holds then we say that the multiplicative hyperring is strongly distributive. 
Let $(\mathbb{Z},+,\cdot)$ be the ring of integers. Corresponding to every subset $X \in P^\star(\mathbb{Z})$ with $\vert X\vert \geq 2$, there exists a multiplicative hyperring $(\mathbb{Z}_X,+,\circ)$ with $\mathbb{Z}_X=\mathbb{Z}$ and for any $a,b\in \mathbb{Z}_X$, $a \circ b =\{a.x.b\ \vert \ x \in X\}$ \cite{das}.
 \begin{definition}\cite{ameri} 
 An element $e$ in $H$ is  an identity element if $a \in a\circ e$ for all $a \in H$. Moreover, an element $e$ in $H$ is  a scalar identity element if $a= a\circ e$ for all $a \in H$.
\end{definition}
Throughout this paper, $H$ denotes a commutative multiplicative hyperring with identity 1. 
\begin{definition} \cite{f10} A non-empty subset $A$ of  $H$ is a  hyperideal  if 
 \begin{itemize}
\item[\rm(i)]~  $a - b \in A$ for all $a, b \in A$; 
\item[\rm(ii)]~ $r \circ a \subseteq A$ for all $a \in A $ and $r \in H$.
 \end{itemize}
 \end{definition}
\begin{definition}
 \cite{das} A proper hyperideal $A$ of   $H$ is called a prime hyperideal if $a \circ b \subseteq A$ for $a,b \in H$ implies that $a \in A$ or $b \in A$. 
 \end{definition} 
  The intersection of all prime hyperideals of $H$ containing a hyperideal $A$ is called the prime radical of $A$,  denoted by $rad (A)$. If the multiplicative hyperring $H$ does not have any prime hyperideal containing $A$, we define $rad(A)=H$. Let {\bf C} be the class of all finite products of elements of $H$ that is ${\bf C} = \{r_1 \circ r_2 \circ . . . \circ r_n \ : \ r_i \in G, n \in \mathbb{N}\} \subseteq P^{\ast }(H)$. A hyperideal $A$ of $H$ is said to be a {\bf C}-hyperideal of $H$ if, for any $J \in {\bf C}, A \cap J \neq \varnothing $ implies $J \subseteq A$.
Let $A$ be a hyperideal of $H$. Then, $D \subseteq rad(A)$ where $D = \{r \in H \ \vert \  r^n \subseteq A \ for \ some \ n \in \mathbb{N}\}$. The equality holds when $A$ is a {\bf C}-hyperideal of $H$ (see Proposition 3.2 in \cite{das}). Recall that a hyperideal $A$ of $H$ is called a strong {\bf C}-hyperideal if for any $E \in \mathfrak{U}$, $E \cap A \neq \varnothing$, then $E \subseteq A$, where $\mathfrak{U}=\{\sum_{i=1}^n J_i \ \vert \ J_i \in {\bf C}, n \in \mathbb{N}\}$ and ${\bf C} = \{r_1 \circ r_2 \circ . . . \circ r_n \ \vert \ r_i \in H, n \in \mathbb{N}\}$ (for more details see \cite{phd}). 
\begin{definition} \cite{ameri} 
A proper hyperideal $A$ of 
 $H$ is maximal in $H$ if for
any hyperideal $B$ of $H$ with $A \subset B \subseteq H$ then $B = H$.
\end{definition}
 Also, we say that $H$ is a local multiplicative hyperring, if it has just one maximal hyperideal.   
\begin{definition} \cite{ameri}  Assume that  $A$ and $B$ are hyperideals of $H$. We define $(A:B)=\{x \in H \ \vert \ x \circ B \subseteq A\}$.
\end{definition}
\section{  $(\alpha, \beta)$-prime hyperideals }
In this section, we introduce the concept of $(\alpha, \beta)$-prime hyperideal  and give some basic properties of them. 
\begin{definition}
Let $\alpha, \beta \in \mathbb{Z}^+$. A proper hyperideal $P$ of  $H$  is called  $(\alpha,\beta)$-prime if $x^{\alpha}\circ y \subseteq P$ for $x,y \in H$ implies $x^{\beta} \subseteq P$ or $y \in P$.
\end{definition}
\begin{example}
Consider the multiplicative hyperring $\mathbb{Z}_A$. Let $p$ is a prime integer such that $A \cap \langle p \rangle = \varnothing$. Then $\langle p \rangle$ is an $(\alpha,\beta)$-prime for all $\alpha,\beta \in \mathbb{Z}^+$.
\end{example}
\begin{remark} \label{close1}
Let $\alpha, \beta \in \mathbb{Z}^+$.
\begin{itemize}
\item[\rm{(i)}]~ Every $(\alpha,\beta)$-prime hyperideal of $H$ is  $(\alpha,\beta)$-closed.
\item[\rm{(ii)}]~ Every $(\alpha,\beta)$-prime hyperideal of $H$ is primary.
\end{itemize}
\end{remark}
The following example shows that the converse part of statements  in Remark \ref{close1}  may not be true, in general.
\begin{example}
Consider the multiplicative hyperring  $(\mathbb{Z}_X,+,\circ)$.
\begin{itemize}
\item[\rm{(i)}]~ Let $X=\{2,3\}$. Then $P=\langle 6 \rangle$ is an $(3,2)$-closed hyperideal of $\mathbb{Z}$, but it is not $(3,2)$-prime as $2^3\circ 3 \subseteq P$ where neither $2^2=\{8,12\} \nsubseteq P$ nor $3 \notin P$.
\item[\rm{(ii)}]~ Let $X=\{2,4\}$. Then $P=\langle 8 \rangle$ is a primary hyperideal of the multiplicative hyperring $(\mathbb{Z},+,\circ)$. However, it is not $(4,3)$-prime as the fact that $1^4\circ4 \subseteq P$ but $1^3=\{4,8,16\} \nsubseteq  P$ and $4 \notin P$.
\end{itemize}
\end{example}
Our first theorem presents a characterization of $(\alpha,\beta)$-prime hyperideals
\begin{theorem} \label{1/1}
Assume that $P \neq H$ is a  hyperideal of $H$ and $\alpha, \beta \in \mathbb{Z}^+$. Then the following are equivalent:
\begin{itemize}
\item[\rm{(i)}]~ $P$ is an $(\alpha,\beta)$-prime hyperideal in $H$.
\item[\rm{(ii)}]~ $P=(P:x^{\alpha})$ such that $x^{\beta} \nsubseteq P$ for  $x \in H$.
\item[\rm{(iii)}]~ If $x^{\alpha} \circ P^{\prime} \subseteq P$ for some hyperideal $P^{\prime}$ of $H$ and $x \in H$, then $x^{\beta} \subseteq P$ or $P^{\prime} \subseteq P$.
\end{itemize}
\end{theorem}
\begin{proof}
(i) $\Longrightarrow$ (ii) Suppose that $x^{\beta} \nsubseteq P$ for  $x \in H$. Take any $y \in (P:x^{\alpha})$. So we have $x^{\alpha} \circ y \subseteq P$. Since $P$ is an $(\alpha,\beta)$-prime hyperideal in $H$ and $x^{\beta} \nsubseteq P$, we get $y \in P$ which means $(P:x^{\alpha}) \subseteq P$. Since the inclusion $P \subseteq (P:x^{\alpha})$ always holds, we obtain $P=(P:x^{\alpha})$. 

(ii) $\Longrightarrow$ (iii) Let $P^{\prime}$ be a hyperideal of $H$ and $x \in H$ such that $x^{\alpha} \circ P^{\prime} \subseteq P$. If $x^{\beta} \subseteq P$, we are done. If $x^{\beta} \nsubseteq P$, we get the result that $ (P:x^{\alpha})=P$ by (ii), and so $P^{\prime} \subseteq P$. 

(iii) $\Longrightarrow$ (i) Let $x^{\alpha} \circ y \subseteq P$ for $x,y \in H$. Then $\langle x^{\alpha} \rangle \circ \langle y \rangle \subseteq  \langle x^{\alpha} \circ y\rangle \subseteq P$ by Proposition 2.15 in \cite{das}. Put $\langle y \rangle =P^{\prime}$. Hence we have $x^{\alpha} \circ P^{\prime} \subseteq P$. From (iii) it follows that either $x^{\beta} \subseteq P$ or $ y \in P^{\prime} \subseteq P$, as needed.
\end{proof}
Recall from \cite{das} that a hyperideal $A$ of $H$ refers to a principal hyperideal if $A=\langle x \rangle$ for $x \in H$. A hyperring whose every hyperideal is principal is called principal hyperideal hyperring.
\begin{proposition} \label{1/2}
Let $P \neq H$ be a hyperideal of a principal hyperideal hyperring $H$ and $\alpha, \beta \in \mathbb{Z}^+$. Then the following statements are equivalent:
\begin{itemize}
\item[\rm{(i)}]~ $P$ is an $(\alpha,\beta)$-prime hyperideal in $H$.
\item[\rm{(ii)}]~ If $P_1^{\alpha} \circ P_2 \subseteq P$ for hyperideals $P_1, P_2$ of $H$, then $P_1^ {\beta} \subseteq P$ or $P_2\subseteq P$. 
\item[\rm{(iii)}]~ $P=(P:P_1^{\alpha})$ such that $P_1^{\beta} \nsubseteq P$ for every hyperideal $P_1$ of $H$.
\item[\rm{(iv)}]~ If $P_1^{\alpha} \circ y \subseteq P$ where $y \in H$ and $P_1$ is a hyperideal of $H$, then $P_1^{\beta} \subseteq P$ or $y \in P$.
\end{itemize}
\end{proposition}
\begin{proof}
(i) $\Longrightarrow$ (ii) Let $P_1^{\alpha} \circ P_2 \subseteq P$ for hyperideals $P_1, P_2$ of $H$. Since $H$ is a principal hyperideal hyperring, there exists $x \in H$ such that $P_1=\langle x \rangle$ and so $x^{\alpha} \circ P_2  \subseteq P$. Since $P$ is an $(\alpha,\beta)$-prime hyperideal in $H$, by Theorem \ref{1/1} we conclude that  $x^{\beta} \subseteq P$ which means $P_1^{\beta}=\langle x \rangle ^{\beta} \subseteq  \langle x^{\beta} \rangle  \subseteq P$  or $P_2 \subseteq P$.

(ii) $\Longrightarrow$ (iii) Let $y \in (P:P_1^{\alpha})$ and $P_1^{\beta} \nsubseteq P$ for a hyperideal $P_1$ of $H$. Then $\langle y \rangle \subseteq  (P:P_1^{\alpha})$ as $(P:P_1^{\alpha})$ is a hyperideal of $H$ by Theoremm 3.8 in \cite{ameri}, and so $P_1^{\alpha} \circ \langle y \rangle \subseteq P$. By the hypothesis, we get $y \in \langle y \rangle \subseteq P$ which implies $(P:P_1^{\alpha}) \subseteq P$. The other containment is clear.

(iii) $\Longrightarrow$ (iv) Let $P_1^{\alpha} \circ y \subseteq P$ and $P_1^{\beta} \nsubseteq P$. Then we have $y \in (P:P_1^{\alpha})=P$, as required.

(iv) $\Longrightarrow$ (i) Let $x^{\alpha} \circ y \subseteq P$ for $x, y \in H$. We assume that $P_1=\langle x \rangle$. Hence we have $P_1^{\alpha} \circ y \subseteq \langle x \rangle^{\alpha} \circ \langle y \rangle \subseteq  \langle x^{\alpha} \circ y \rangle \subseteq P$. By the assumption, we get the result that $x^{\beta} \subseteq P_1 ^{\beta} \subseteq P$ or $y \in P$. Consequently, $P$ is an $(\alpha,\beta)$-prime hyperideal in $H$.
\end{proof}
By Remark \ref{close1}, every $(\alpha,\beta)$-prime $\mathcal{C}$-hyperideal of $H$ is a primary hyperideal and so its radical is a prime hyperideal of $H$ by Proposition 3.6 in \cite{das}. Let $P$ be an $(\alpha,\beta)$-prime $\mathcal{C}$-hyperideal of $H$. Then  $P$ is referred as a $Q$-$(\alpha,\beta)$-prime $\mathcal{C}$-hyperideal of $H$ where $rad(P)=Q$.
\begin{theorem} \label{1}
If $P$ is a $Q$-$(\alpha,\beta)$-prime $\mathcal{C}$-hyperideal of $H$ for $\alpha, \beta \in \mathbb{Z}^+$, then $Q=\{x \in H \ \vert \ x^{\beta} \subseteq P\}$.
\end{theorem}
\begin{proof}
The inclusion $\{x \in H \ \vert \ x^{\beta} \subseteq P\} \subseteq Q$  holds. Assume that $x \in Q=rad(P)$ and $n$ is the smallest positive integer with $x^n \subseteq P$. Then we have $x^{\alpha} \circ x ^{n-1} \subseteq P$. Take any $y \in x^{n-1}$. Since  $x^{\alpha} \circ y \subseteq P$ and $P$ is an $(\alpha,\beta)$-prime hyperideal of $H$, we get $x^{\beta} \subseteq P$ or $y \in P$. Let  $y \in P$. Since $P$ is a $\mathcal{C}$-hyperideal of $H$ and $y \in x^{n-1}$, we get the result that  $x^{n-1} \subseteq P$, a contradiction. Therefore we obtain $x^{\beta} \subseteq P$. Then $Q \subseteq \{x \in H \ \vert \ x^{\beta} \subseteq P\}$, so this completes the proof.
\end{proof}
\begin{theorem} \label{2}
Let  $P$ be a $\mathcal{C}$-hyperideal of $H$ with an $i$-set and $\alpha,\beta \in \mathbb{Z}^+$ such that  $Q=\{x \in H \ \vert \ x^{\beta} \subseteq P\}$ is a maximal hyperideal of $H$. Then $P$ is an $Q$-$(\alpha,\beta)$-prime hyperideal. 
\end{theorem}
\begin{proof}
Assume that $x^{\alpha} \circ y \subseteq P$ with $x^{\beta} \nsubseteq P$ for $x,y \in H$. Let $x^{\alpha} \subseteq Q$. Since $Q$ is a maximal hyperideal of $H$, we get $x \in Q$ by Proposition 2.18 in \cite{das}.  This contradict by $x^{\beta} \nsubseteq P$. Then $x^{\alpha} \nsubseteq Q$. Let  $a \in x^{\alpha}$ such that $a \notin Q$.  Then we have $\langle a,Q \rangle =H$. So, there exists $m \in Q$ such that $1 \in \langle a \rangle+m \subseteq \langle x^{\alpha} \rangle+m$. Therefore $1 \in ( \langle x^{\alpha} \rangle+m)^{\beta} \subseteq  \Sigma_{i=0}^{\beta} \tbinom{\beta}{i} \langle x^{\alpha} \rangle^{\beta-i} \circ m^i$ and so  $y \in 1 \circ y  \subseteq  (\Sigma_{i=0}   \tbinom{\beta}{i} \langle x^{\alpha} \rangle^{\beta-i} \circ m^i) \circ y \subseteq P$. Hence $P$ is an $(\alpha,\beta)$-prime hyperideal. By Theorem \ref{1}, we conclude that $P$ is an $Q$-$(\alpha,\beta)$-prime hyperideal.
\end{proof}
As an immediate consequence of the previous theorem, we have the following
result.
\begin{corollary}\label{3}
Let $Q$ is a maximal hyperideal of $H$, $\alpha,\beta \in \mathbb{Z}^+$ and $n \leq \beta$. Then  $Q^n$ is an $Q$-$(\alpha,\beta)$-prime hyperideal.
\end{corollary}
\begin{theorem}\label{4}
Assume that  $P_1,\cdots,P_n$ are $\mathcal{C}$-hyperideals of $H$ such that for every  $i \in \{1,\cdots,n\}$, $P_i$ is a  $Q$-$({\alpha}_i,{\beta}_i)$-prime hyperideal with ${\alpha}_i,{\beta}_i \in \mathbb{Z}^+$. Then $\bigcap_{i=1}^n P_i$ is a $Q$-$(\alpha,\beta)$-prime hyperideal in $H$ where $\alpha \leq \min \{{\alpha}_1,\cdots, {\alpha}_n\}$ and $\beta \geq \max \{{\beta}_1,\cdots,{\beta}_n\}$.
\end{theorem}
\begin{proof}
Suppose that $x^{\alpha} \circ y \subseteq \bigcap_{i=1}^n P_i$ for $x,y \in H$ such that $y \notin \bigcap_{i=1}^n P_i$. This means  $y \notin  P_t$ for some  $t \in \{1,\cdots,n\}$. Since $ P_t$ is a $Q$-$({\alpha}_t,{\beta}_t)$-prime hyperideal in $H$ and $x^{{\alpha}_t} \circ y \subseteq P_t$, we conclude that  $x^{{\beta}_t} \subseteq P_t$ which implies $x \in Q$. Then we get $ x^{{\beta}_i} \subseteq P_i$ for each $i \in \{1,\cdots,n\}$ by Theorem \ref{1}, and so $x^{\beta} \subseteq \bigcap_{i=1}^n P_i$ where $\beta \geq \max \{{\beta}_1,\cdots,{\beta}_n\}$. On the other hand, we have $rad(\bigcap_{i=1}^n P_i)=\bigcap_{i=1}^n rad(P_i)=Q$ by Proposition 3.3 in \cite{das}. Consequently, $\bigcap_{i=1}^n P_i$ is a $Q$-$(\alpha,\beta)$-prime hyperideal in $H$.
\end{proof}
Let   $\mathfrak{C}(P)=\{(\alpha,\beta) \in \mathbb{Z}^+ \times \mathbb{Z}^+ \ \vert \ P \text{ is ($\alpha$,$\beta$)-closed }\}$ where $P$ is a proper hyperideal  of $H$. Then we get $\{(\alpha,\beta) \in \mathbb{Z}^+ \times \mathbb{Z}^+ \ \vert \ 1 \leq \alpha \leq \beta \} \subseteq \mathfrak{C}(P) \subseteq \mathbb{Z}^+ \times \mathbb{Z}^+$.  Furthermore, $rad(P)=P$ if and only if $\mathfrak{C}(P)=\mathbb{Z}^+ \times \mathbb{Z}^+$ \cite{anb7}. Now, let us  define  $\mathfrak{L}(P)=\{(\alpha,\beta) \in \mathbb{Z}^+ \times \mathbb{Z}^+ \ \vert \ P \text{ is ($\alpha$,$\beta$)-prime }\}$. Let $\mathfrak{L}(P)=\mathbb{Z}^+ \times \mathbb{Z}^+$. Clearly, the hyperideal $P$ is prime if and only if $(1,1) \in \mathfrak{L}(P)$. 
 \begin{remark} \label{5}
 Let $P$ be a proper hyperideal of $H$ and $\alpha,\beta \in \mathbb{Z}^+$. If $(\alpha,\beta) \in \mathfrak{L}(P)$, then $(\alpha^{\prime},\beta^{\prime}) \in \mathfrak{L}(P)$ for $\alpha^{\prime},\beta^{\prime} \in \mathbb{Z}^+$ such that $ \alpha^{\prime} \leq \alpha$ and $\beta \leq \beta^{\prime}$ by being $P$ a hyperideal.
 \end{remark}
  \begin{remark}
Suppose that $P$ is a proper hyperideal of $H$ and $\alpha,\beta,\theta \in \mathbb{Z}^+$. If $(\alpha,\beta) \in \mathfrak{L}(P)$ and $(\beta,\theta) \in \mathfrak{C}(P)$, then $(\alpha,\theta) \in \mathfrak{L}(P)$.
 \end{remark}
 \begin{theorem}\label{6}
 Assume that $P$ is a proper hyperideal of $H$ and $\alpha,\beta \in \mathbb{Z}^+$. Then $(\alpha,\beta) \in \mathfrak{L}(P)$ if and only if  $(\alpha+1,\beta) \in \mathfrak{L}(P)$.
 \end{theorem}
 \begin{proof}
 $(\Longrightarrow)$ Let $(\alpha,\beta) \in \mathfrak{L}(P)$. Assume that $x^{\alpha+1} \circ y \subseteq P$ for $x,y \in H$ such that $y \notin P$. Since $\alpha+1 \leq 2\alpha$, we have $(x^2)^{\alpha} \circ y \subseteq P$. Since $(\alpha,\beta) \in \mathfrak{L}(P)$ and $y \notin P$, we get $x^{2\beta} \subseteq P$ which means $x \in rad(P)$. By Theorem \ref{1}, we conclude that $x^{\alpha} \subseteq P$.
  $(\Longleftarrow)$ It follows from Remark \ref{5}.
 \end{proof}
 \begin{theorem}
 Let the zero hyperideal of $H$ be a $\mathcal{C}$-hyperideal and $\alpha,\beta \in \mathbb{Z}^+$. Then every proper hyperideal of $H$ is $(\alpha,\beta)$-prime if and only if $H$ has no non-trivial idempotents, every prime $\mathcal{C}$-hyperideal of $H$ is maximal.
 \end{theorem}
 \begin{proof}
 Let every proper hyperideal of $H$ be $(\alpha,\beta)$-prime. Assume that $e$ is a non-trivial idempotent in $H$. Then  $0 \in e \circ (e-1)$ and so $0 \in e^{\alpha} \circ (e-1)$. Since the zero hyperideal of $H$ is a $\mathcal{C}$-hyperideal, we have $ e^{\alpha} \circ (e-1)=0$. Therefore we get $e \in e^{\beta}=0$ as the zero hyperideal of $H$ is $(\alpha,\beta)$-prime and $e \neq 1$.  This is a contradiction and so $H$ has no non-trivial idempotents. Suppose that  $P$ is a  prime $\mathcal{C}$-hyperideal that is not maximal. Then we have $P \subset Q$ for some hyperideal $Q$ of $H$. Let $x \in Q-P$. Then $x^{\alpha} \circ x \subseteq  \langle x^{\alpha+1} \rangle$. This follows that $x^{\beta} \subseteq  \langle x^{\alpha+1} \rangle$ or $x \in  \langle x^{\alpha+1} \rangle$. This implies that $x^{\beta} \subseteq x^{\alpha+1} \circ a $ for some $a \in H$ or $x \in x^{\alpha+1} \circ b $ for some $b \in H$. In the first case, $0 \in x^{\beta}-x^{\alpha+1} \circ a$. Since $P$ is a $\mathcal{C}$-hyperideal, we get $x^{\beta}-x^{\alpha+1} \circ a \subseteq P$. From $x^{\beta} \circ (1-x^{\alpha-\beta+1} \circ r) \subseteq x^{\beta}-x^{\alpha+1} \circ a$ it follows that $x^{\beta} \circ (1-x^{\alpha-\beta+1} \circ r) \subseteq P$. Since $P$ is a prime hyperideal and $x^{\beta} \nsubseteq P$, we get $1-x^{\alpha-\beta+1} \circ r \subseteq P \subset Q$ and so $1 \in Q$, a contradiction. In the second case, we get $0 \in (x-x^{\alpha+1} \circ b) \cap P$. Then we get the result that $x \circ (1-x^{\alpha} \circ b) \subseteq x-x^{\alpha+1} \circ b \subseteq  P$. Thus we have $1-x^{\alpha} \circ b \subseteq P \subset Q$ because $x \nsubseteq P$. This follows that $1 \in Q$ which is a contradiction. Consequently, every prime $\mathcal{C}$-hyperideal of $H$ is maximal. 
 \end{proof}
 We say that a hyperideal $P$ of $H$ is of maximum length $\beta$ if for every ascending chain $P=P_0 \subseteq P_1 \subseteq P_2 \subseteq \cdots$ of hyperideals of $H$, $\beta$ is the largest integer with $P_{\beta}=P_{\beta+1}=\cdots$.
 \begin{theorem}
Assume that  $\alpha,\beta \in \mathbb{Z}^+$ and $P$ is a  strong $\mathcal{C}$-hyperideal of $H$ of maximum length $\beta$. If $P$ is irreducible in $H$, then $P$ is an $(\alpha,\beta)$-prime hyperideal.
 \end{theorem}
 \begin{proof}
 Let $x^{\alpha} \circ y \subseteq P$ for $x,y \in H$. Consider the  ascending chain $P=P_0 \subseteq P_1 \subseteq P_2 \subseteq \cdots$ where  $P_i=\{b \in H \ \vert \ x^i \circ b \subseteq P\}$. By the assumption we have $P_{\beta}=P_{\beta+1}=\cdots$. Put $I=P+\langle x^{\beta} \rangle$ and $J=P+\langle y \rangle$. Then we have $P \subseteq I \cap J$. Let $a \in I \cap J$. Then there exist $a_1, a_2 \in H$ and $p_1, p_2 \in P$ such that $a \in (p_1+x^{\beta} \circ a_1) \cap (p_2+y\circ a_2).$ It follows that $a=p_1+x_1=p_2+x_2$ for some $x_1 \in x^{\beta} \circ a_1$ and $x_2 \in y\circ a_2$. Therefore we get $x_1-x_2 \in (x^{\beta} \circ a_1 -y\circ a_2) \cap P$ and so $ x^{\beta} \circ a_1 -y\circ a_2 \subseteq P$ as $P$ is a strong $\mathcal{C}$-hyperideal of $H$. Since $ (x^{\beta} \circ a_1 -y\circ a_2 ) \circ x^{\alpha} \subseteq x^{\alpha+\beta} \circ a_1 \ -x^{\alpha} \circ y\circ a_2$, we obtain $x^{\alpha+\beta} \circ a_1 \ -x^{\alpha} \circ y\circ a_2 \subseteq P$ which implies $x^{\alpha+\beta} \circ a_1 \subseteq P$. This means $a_1 \in P_{\alpha+\beta}=P_{\beta}$ and so $x^{\beta} \circ a_1 \subseteq P$. Then we conclude that $a=p_1+x_1 \in P$ which means $I \cap J \subseteq P$ and so $P=I \cap J$. By the hypothesis, we get $P=I$ which implies $x^\beta \subseteq P$ or $P=J$ which means $y \in P$. Thus $P$ is an $(\alpha,\beta)$-prime hyperideal.
 \end{proof}

 Let $S$ be a non-empty subset of $H$ with scalar identity $1$.  Recall from \cite{ameri} that $S$ is a multiplicative closed subset (briefly, MCS), if  $S$ is closed under the hypermultiplication and $S$ contains $1$. \cite{Mena} Consider the set $(H \times S / \sim)$ of equivalence classes, being denoted by  $S^{-1}H$, such that $(x_1,t_1) \sim (x_2,t_2)$ if and only if there exists $ t \in S $ with  $ t \circ t_1 \circ x_2=t \circ t_2 \circ x_1$.
The equivalence class of $(x,t) \in H \times S$ is denoted by $\frac{x}{t}$. The set $S^{-1}H$ is a multiplicative hyperring where the operation  $\oplus$ and the  multiplication $\odot$ are
defined by 

$\hspace{1cm}\frac{x_1}{t_1} \oplus \frac{x_2}{t_2}=\frac{t_1 \circ x_2+t_2 \circ x_1}{t_1 \circ t_2}=\{\frac{a+b}{c} \ \vert \ a \in t_1 \circ x_2 , b \in t_2 \circ x_1, c \in t_1 \circ t_2\}$

$\hspace{1cm}\frac{x_1}{t_1} \odot \frac{x_2}{t_2}=\frac{x_1 \circ a_2}{t_1 \circ t_2}=\{\frac{a}{b} \ \vert \ a \in x_1 \circ x_2, b \in t_1 \circ t_2\}$

The localization map $\pi: H \longrightarrow S^{-1}H$, defined by $x \mapsto \frac{x}{1}$,  is a homomorphism of hyperrings. Also, if $A$ is a hyperideal of $H$, then $S^{-1}A$ is a hyperideal of $S^{-1}H$ \cite{Mena}.
\begin{theorem}
Assume that $P$ is a $\mathcal{C}$-hyperideal of $H$ and $S$ a  MCS such that  $P \cap S = \varnothing$. If $P$ is an $(\alpha,\beta)$-prime hyperideal of $H$, then $S^{-1}P$ is an $(\alpha,\beta)$-prime hyperideal of $S^{-1}H$. 
\end{theorem}
\begin{proof}
Let $\underbrace{\frac{x}{t_1} \odot \cdots \odot \frac{x}{t_1}}_{\alpha}\odot \frac{y}{t_2}=\frac{x^{\alpha} \circ y}{t_1^{\alpha} \circ t_2} \subseteq S^{-1}P$ for some $\frac{x}{t_1}, \frac{y}{t_2}\in S^{-1}H$. So we get $\frac{a}{t} \in \frac{x^{\alpha}\circ y}{t_1^{\alpha}\circ t_2}$ for every $a \in x^{\alpha} \circ y=\underbrace{x \circ \cdots \circ x}_{\alpha} \circ y$ and $t \in t_1^{\alpha} \circ t_2=\underbrace{t_1 \circ \cdots \circ t_1}_{\alpha} \circ t_2$. Hence $\frac{a}{t}=\frac{a^{\prime}}{t^{\prime}}$ for some $a^{\prime} \in P$ and $t^{\prime} \in S$. Then there exists $s \in S$ such that $s \circ a \circ t^{\prime}=s \circ a^{\prime} \circ t$. This implies $s \circ a \circ t^{\prime} \subseteq P$. Since $a \in x^{\alpha} \circ y$, we conclude that $s \circ a \circ t^{\prime} \subseteq s \circ x^{\alpha}\circ y \circ t^{\prime}$. Since $P$ is a $\mathcal{C}$-hyperideal of $H$, we get  $s \circ x^{\alpha} \circ y \circ t^{\prime} \subseteq P$  and then $s^{\alpha} \circ x^{\alpha} \circ y \circ {t^{\prime}}^{\alpha}=(s \circ x \circ t^{\prime})^{\alpha} \circ y \subseteq P$.  Take any $z \in s \circ x \circ t^{\prime}$. Since $z^{\alpha} \circ y \subseteq P$ and $P$ is an $(\alpha,\beta)$-prime hyperideal of $H$, we have either $z^{\beta} \subseteq P$ or $y \in P$. In first possibily, we get $(s \circ x \circ t^{\prime})^{\beta} \subseteq P$ as $P$ is a $\mathcal{C}$-hyperideal of $H$ and $z^{\beta} \subseteq (s \circ x \circ t^{\prime})^{\beta}$. Hence $\frac{x^{\beta}}{t_1^{\beta}}=\frac{s^{\beta} \circ x^{\beta} \circ {t^{\prime}}^{\beta}}{s^{\beta} \circ t_1^{\beta} \circ {t^{\prime}}^{\beta}} \subseteq S^{-1}P$ which means $\underbrace{\frac{x}{t_1} \odot \cdots \odot \frac{x}{t_1}}_{\beta} \subseteq S^{-1}P$ or $\frac{y}{t_2} \in S^{-1}P$. This shows that  $S^{-1}P$ is an $(\alpha,\beta)$-prime hyperideal of $S^{-1}H$.
\end{proof}

Assume that $(H,+,\circ)$ is a multiplicative hyperring and $x$ is an indeterminate. Then $(H[x],+,\diamond)$ is a polynomail multiplicative hyperring such that $ux^n \diamond vx^m=(u \circ v )x^{n+m}$ \cite{Ciampi}.
\begin{theorem} \label{polynomail}
Let  $\alpha, \beta \in \mathbb{Z}^+$. If $P$ is an $(\alpha,\beta)$-prime hyperideal of $(H,+,\circ)$, then $P[x]$ is an $(\alpha,\beta)$-prime hyperideal of $(H[x],+,\diamond)$. 
\end{theorem}
\begin{proof}
Suppose that $u(x)^{\alpha} \diamond v(x) \subseteq P[x]$. Without
loss of generality, we may assume that $u(x)=ax^n$ and $v(x)=bx^m$ for $a,b \in H$. Hence $a^{\alpha} \circ b x^{\alpha n+m} \subseteq P[x]$. This means $a^{\alpha} \circ b \subseteq P$. Since $P$ is an $(\alpha,\beta)$-prime hyperideal of $(H,+,\circ)$, we get $a^{\beta} \subseteq P$ or $b \in P$ which implies $u(x)^{\beta}={(ax^n)}^{\beta} =a^{\beta}x^{\beta \cdot n}\subseteq P$ or $v(x)=bx^m \in P[x]$. Consequently, $P[x]$ is an $(\alpha,\beta)$-prime hyperideal of $(H[x],+,\diamond)$. 
\end{proof}
In view of Theorem \ref{polynomail}, we have the following result.
\begin{corollary}
Assume that $P$ is an $(\alpha,\beta)$-prime hyperideal of $H$. Then $P[x]$ is an $(\alpha,\beta)$-prime hyperideal of $H[x]$. 
\end{corollary}
  
Assume that $H$ is a multiplicative hyperring. Then the set of all hypermatrices of $H$ is denoted by  $M_m(H)$. Let  $A = (A_{ij})_{m \times m}, B = (B_{ij})_{m \times m} \in P^\star (M_m(H))$. Then $A \subseteq B$ if and only if $A_{ij} \subseteq B_{ij}$\cite{ameri}. 
\begin{theorem} \label{8} 
Suppose that  $P$ is a hyperideal of $H$ and $\alpha, \beta \in \mathbb{Z}^+$. If $M_m(P)$ is an $(\alpha,\beta)$-prime hyperideal of $M_m(H)$, then $P$ is an $(\alpha,\beta)$-prime hyperideal of $H$. 
\end{theorem}
\begin{proof}
Let $x^{\alpha} \circ y \subseteq P$ for some $x,y \in H$. Then we have
\[\begin{pmatrix}
x^{\alpha}\circ y & 0 & \cdots & 0 \\
0 & 0 & \cdots & 0 \\
\vdots & \vdots & \ddots \vdots \\
0 & 0 & \cdots & 0 
\end{pmatrix}
\subseteq M_m(P).\]
Since   $M_m(P)$ is an $(\alpha,\beta)$-prime hyperideal of $M_m(H)$ and 
\[ \begin{pmatrix}
x^{\alpha}\circ y & 0 & \cdots & 0\\
0 & 0 & \cdots & 0\\
\vdots & \vdots & \ddots \vdots\\
0 & 0 & \cdots & 0
\end{pmatrix}
=
\underbrace{\begin{pmatrix}
x & 0 & \cdots & 0\\
0 & 0 & \cdots & 0\\
\vdots & \vdots & \ddots \vdots\\
0 & 0 & \cdots & 0
\end{pmatrix}
\circ \cdots \circ
\begin{pmatrix}
x & 0 & \cdots & 0\\
0 & 0 & \cdots & 0\\
\vdots & \vdots & \ddots \vdots\\
0 & 0 & \cdots & 0
\end{pmatrix}}_{\alpha}\circ
\begin{pmatrix}
y & 0 & \cdots & 0\\
0 & 0 & \cdots & 0\\
\vdots & \vdots & \ddots \vdots\\
0 & 0 & \cdots & 0
\end{pmatrix},
\]
we get the result that 
\[ \underbrace{\begin{pmatrix}
x & 0 & \cdots & 0\\
0 & 0 & \cdots & 0\\
\vdots& \vdots & \ddots \vdots\\
0 & 0 & \cdots & 0
\end{pmatrix} 
\circ \cdots \circ
\begin{pmatrix}
x & 0 & \cdots & 0\\
0 & 0 & \cdots & 0\\
\vdots& \vdots & \ddots \vdots\\
0 & 0 & \cdots & 0
\end{pmatrix}}_{\beta}=
\begin{pmatrix}
x^{\beta} & 0 & \cdots & 0\\
0 & 0 & \cdots & 0\\
\vdots& \vdots & \ddots \vdots\\
0 & 0 & \cdots & 0
\end{pmatrix}
\subseteq M_m(P)\]
or 
\[\begin{pmatrix}
y & 0 & \cdots & 0\\
0 & 0 & \cdots & 0\\
\vdots & \vdots & \ddots \vdots\\
0 & 0 & \cdots & 0
\end{pmatrix} \in M_m(P).
\]

Then we conclude that  $x^{\beta} \subseteq P$ or $y \in P$. Consequently,  $P$ is an $(\alpha,\beta)$-prime hyperideal of $H$.
\end{proof}
 Recall from \cite{f10} that a mapping $\psi$ from the multiplicative hyperring
$(H_1, +_1, \circ _1)$ into the multiplicative hyperring $(H_2, +_2, \circ _2)$  refers to  a hyperring good homomorphism if $\psi(a +_1 b) =\psi(a)+_2 \psi(b)$ and $\psi(a\circ_1b) = \psi(a)\circ_2 \psi(b)$  for all $a,b \in H_1$.

\begin{theorem} \label{homo} 
Assume that $H_1$ and $H_2$ are two multiplicative hyperrins, $\psi: H_1 \longrightarrow H_2$ a hyperring
good homomorphism and $\alpha, \beta \in \mathbb{Z}^+$. 
\begin{itemize}
\item[\rm{(i)}]~ If $P_2$ is an $(\alpha,\beta)$-prime hyperideal of $H_2$, then $\psi^{-1}(P_2)$ is an $(\alpha,\beta)$-prime hyperideal of $H_1$.
\item[\rm{(ii)}]~ If $\psi$ is surjective and $P_1$ is a an $(\alpha,\beta)$-prime $\mathcal{C}$-hyperideal of $H_1$ with $Ker (\psi) \subseteq P_1$ , then $\psi(P_1)$ is an $(\alpha,\beta)$-prime hyperideal of $H_2$.
\end{itemize}
\end{theorem}
\begin{proof}
(i) Let $ x^{\alpha} \circ_1 x_1 \subseteq \psi^{-1}(P_2)$ for some $x,x_1 \in H_1$. Then we have $ \psi(x^{\alpha} \circ_1 x_1)=\psi(x)^{\alpha} \circ_2 \psi(x_1) \subseteq P_2$ as $\psi$ is  a hyperring
good homomorphism. Since $P_2$ is an $(\alpha,\beta)$-prime hyperideal of $H_2$, we have $\psi(x^{\beta})=(\psi(x))^{\beta} \subseteq P_2$   which means $x^{\beta} \subseteq \psi^{-1}(P_2)$ or $\psi(x_1) \in P_2$ which implies $x_1 \in \psi^{-1}(P_2)$. Consequently,   $\psi^{-1}(P_2)$ is an $(\alpha,\beta)$-prime hyperideal of $H_1$.

(ii)  Let $y^{\alpha} \circ_2 y_1 \subseteq \psi(P_1)$ for  $y,y_1 \in H_2$. Then  $\psi(x)=y$ and $\psi(x_1)=y_1$for some $x,x_1 \in H_1$ because $\psi$ is surjective. Hence $\psi(x^{\alpha} \circ_1 x_1)=\psi(x)^{\alpha} \circ_2 \psi(x_1)\subseteq \psi(P_1)$. Now, pick any $a \in x^{\alpha} \circ_1 x_1$. Then $\psi(a) \in \psi(x^{\alpha} \circ_1 x_1) \subseteq \psi(P_1)$ and so there exists $b \in P_1$ such that $\psi(a)=\psi(b)$. Then we have $\psi(a-b)=0$ which means $a-b \in Ker (\psi)\subseteq P_1$ and so  $a \in P_1$. Therefore $x^{\alpha} \circ_1 x_1  \subseteq P_1$ as $P_1$ is a $\mathcal{C}$-hyperideal. Since $P_1$ is an $(\alpha,\beta)$-prime hyperideal of $H_1$,  we obtain  $x^{\beta} \subseteq P_1$ or $x_1 \in P_1$. This implies that $y^{\beta} =\psi(x^{\beta})\subseteq \psi(P_1)$ or $y_1=\psi(x_1) \in \psi(P_1)$.  Thus  $\psi(P_1)$  is an $(\alpha,\beta)$-prime hyperideal of $H_2$.
\end{proof}
Now, we have the following result.
\begin{corollary}
 Let $P_1$ and $P_2$ be two hyperideals of $H$ with $P_1 \subseteq P_2$ and $\alpha,\beta \in \mathbb{Z}^+$. Then  $P_2$ is an $(\alpha,\beta)$-prime hyperideal of $H$ if and only if   $P_2/P_1$ is an $(\alpha,\beta)$-prime hyperideal of $H/P_1$.
\end{corollary}
\begin{proof}
Consider the homomorphism $\pi :H \longrightarrow H/P_1$ defined by $\pi(x)=x+P_1$. Then the claim follows from Theorem \ref{homo} as $\pi$ is a good epimorphism.
\end{proof}
 \section{weakly $(\alpha,\beta)$-prime hyperideals}
 \begin{definition}
Assume that $P$ is a proper hyperideal of $H$ and $\alpha,\beta\in \mathbb{Z}^+$. $P$ is said to be  a weakly $(\alpha,\beta)$-prime hyperideal if  $0 \notin x^{\alpha} \circ y \subseteq P$ for $x,y \in H$ implies  that $x^{\beta} \subseteq P$ or $y \in P$. 
 \end{definition}
 \begin{example}
Consider the ring $(\mathbb{Z}_8,\oplus, \odot)$ where $\bar{x} \oplus \bar{y}$ and $\bar{x} \odot \bar{y}$ are remainder of  $\frac{x+y}{8}$ and $\frac{x \cdot y}{8}$, respectively,  where $+$ and $\cdot$ are ordinary addition and multiplication for all $\bar{x}, \bar{y} \in \mathbb{Z}_8$. Define the hyperoperation $\bar{x} \circ \bar{y}=\{\overline{xy},\overline{2xy}, \overline{3xy},\overline{4xy},\overline{5xy},\overline{6xy},\overline{7xy}\}$. Then the hyperideal $Q=\{\bar{0},\bar{4}\}$ of $(\mathbb{Z}_8,\oplus,\circ)$ is weakly $(3,1)$-prime but it is not $(3,1)$-prime. 
\end{example}
 \begin{theorem}\label{8}
Assume that the zero hyperideal of $H$ is  a $\mathcal{C}$-hyperideal such that its radical is  prime. If $P$ is a weakly  $(\alpha,\beta)$-prime $\mathcal{C}$-hyperideal of $H$ for $\alpha,\beta\in \mathbb{Z}^+$, then $rad(P)$ is prime. In particular, $x^{\beta} \subseteq P$ for each $x \in rad(P)-rad(0)$.
 \end{theorem}
 \begin{proof}
 Let $x \circ y \subseteq rad(P)$ for $x,y \in H$. Then there exists $n \in \mathbb{Z}^+$ such that $x^n \circ y^n \subseteq P$. Assume that $a \in x^n$ and $b \in y^n$, so $a^{\alpha} \circ b \subseteq P$. If $0 \in a^{\alpha} \circ b$, then $a^{\alpha} \circ b =0$ as the zero hyperideal of $H$ is  a $\mathcal{C}$-hyperideal. Since radical of the hyperideal is prime, we have $a \in rad(0)$ or $b \in rad(0)$. This means $x \in rad(0)$ or $y \in rad(0)$. Then we get the result that  $x \in rad(P)$ or $y \in rad(P)$. Assume that $0 \notin a^{\alpha} \circ b$. Since $P$ is a weakly  $(m,n)$-prime hyperideal of $H$, we get either $a^{\beta} \subseteq P$ or $b \in P$. Since $P$ is a $\mathcal{C}$-hyperideal, $a^{\beta} \subseteq x^{\beta n}$ or $b \in y^n$, we have $x^{\beta n} \subseteq P$ which implies $x \in rad(P)$ or $y^n \subseteq P$ which means $y \in rad(P)$, as needed.
Now, take any $x \in rad(P)-rad(0)$. Assume that $n$ is the least positive integer with $x^n \subseteq P$. Since $x \notin rad(0)$, we conclude that $0 \notin x^{\alpha} \circ x^{n-1}$. Let $y \in x^{n-1}$. So $0 \notin x^{\alpha} \circ y \subseteq P$. Since $P$ is a weakly  $(m,n)$-prime hyperideal and $y \notin P$, we get $x^{\beta} \subseteq P$, as required.
 \end{proof}
 \begin{theorem}\label{9}
  If every proper hyperideal of $H$ is weakly $(\alpha,\beta)$-prime such that  $\alpha,\beta\in \mathbb{Z}^+$ and $\alpha \geq \beta$, then every prime $\mathcal{C}$-hyperideal of $H$ is maximal.
 \end{theorem}
\begin{proof}
Suppose that  $I $ is a prime $\mathcal{C}$-hyperideal such that it is not maximal.  Let $J$ be a proper hyperideal such that $I \subset J$. Take any $x \in J-I$. Put $P=\langle x^{\alpha+1} \rangle$. Therefore we have $0 \notin x^{\alpha} \circ x \subseteq P$. By the hypothesis, we have $x^{\beta} \subseteq P$ or $x \in P$. In the first possibilty, we obtan $x^{\beta} \subseteq x^{\alpha+1} \circ r$ for some $r \in H$  which means $0 \in x^{\beta} - x^{\alpha+1} \circ r \cap I$. Then $x^{\beta} - x^{\alpha+1} \circ r  \subseteq I$, also $x^{\beta} \circ (1-x^{\alpha-\beta+1}\circ r) \subseteq x^{\beta} - x^{\alpha+1} \circ r $. Hence $x^{\beta} \circ (1-x^{\alpha-\beta+1}\circ r) \subseteq I$. Since $x^{\beta} \nsubseteq  I$ and $I $ is a prime hyperideal, we get $1-x^{\alpha-\beta+1}\circ r \subseteq I \subset J$ which means $1 \in J$, a contradiction. In the second possibilty, we get a contradiction by a similar argument. Consequenntly, every prime $\mathcal{C}$-hyperideal of $H$ is maximal.
\end{proof}
 Let $P$ be a weakly $(\alpha,\beta)$-prime $\mathcal{C}$-hyperideal of $H$ and $x,y \in H$. We say that $(x,y)$ is an $(\alpha,\beta)$-zero of $P$ if $0 \in x^{\alpha} \circ y$, $x^{\beta} \nsubseteq P$ and $y \notin P$.
 \begin{proposition}\label{10}
 Let $P$ be a weakly $(\alpha,\beta)$-prime $\mathcal{C}$-hyperideal of $H$ and  and $(x,y)$ be an $(\alpha,\beta)$-zero of $P$ where $\alpha,\beta\in \mathbb{Z}^+$. Then the following hold: 
 \begin{itemize}
\item[\rm(i)]~$0 \in (x+a)^{\alpha} \circ y$ for all $a \in P$.
\item[\rm(ii)]~ $0 \in x^{\alpha} \circ (y+a)$ for all $a \in P$.
\item[\rm(iii)]~ If the hyperideal zero of $H$ is a strong $\mathcal{C}$-hyperideal, then $ x^{\alpha} \circ a=0$ for all $a \in P$.
 \end{itemize}
 \end{proposition}
 \begin{proof}
(i) Let  $(x,y)$ be an $(\alpha,\beta)$-zero of $P$ and $0 \notin (x+a)^{\alpha} \circ y$ for some $a \in P$. Since $0 \in x^{\alpha} \circ y$ and $P$ is a $\mathcal{C}$-hyperideal of $H$, we conclude that $ x^{\alpha} \circ y \subseteq P$. Therefore $0 \notin (x+a)^{\alpha} \circ y \subseteq x^{\alpha} \circ y +\Sigma_{i=1}^{\alpha} \tbinom{\alpha}{i}x^{\alpha-i}\circ a^i\circ y \subseteq P.$ Since $P$ is a weakly $(\alpha,\beta)$-prime hyperideal of $H$ and $y \notin P$, we get $(x+a)^{\beta} \subseteq P$. On the other hand, since $(x,y)$ is an $(\alpha,\beta)$-zero of $P$ and $x^{\beta} \nsubseteq P$, we get the result that $(x+a)^{\beta} \nsubseteq P$ which is a contradiction. Thus $0 \in (x+a)^{\alpha} \circ y$ for all $a \in P$.

(ii) Let $0 \notin x^{\alpha} \circ (y+a)$ for some  $a \in P$. Hence $0 \notin x^{\alpha} \circ (y+a) \subseteq x^{\alpha} \circ y+x^{\alpha} \circ a \subseteq P$ as $P$ is a $\mathcal{C}$-hyperideal of $H$. Since $P$ is a weakly $(\alpha,\beta)$-prime hyperideal of $H$ and $x^{\beta} \nsubseteq P$, we obtain $y+a \in P$ and so $y \in P$ which is a contradiction. Consequently,  $0 \in x^{\alpha} \circ (y+a)$ for all $a \in P$.

(iii) Assume that $x^{\alpha} \circ a \neq 0$ for some $a \in P$. Then there exists $u \in x^{\alpha} \circ a$ such that $u \neq 0$.  By (ii) we have $0 \in x^{\alpha} \circ (y+a)$. Since the hyperideal zero of $H$ is a strong $\mathcal{C}$-hyperideal and $0 \in x^{\alpha} \circ (y+a)\subseteq x^{\alpha} \circ y+x^{\alpha} \circ a$, we get $x^{\alpha} \circ y+x^{\alpha} \circ a=0$. Moreover,  since $0 \in x^{\alpha} \circ y$ and $0 \neq u \in x^{\alpha} \circ a$, we get $u = 0+u \in x^{\alpha} \circ y+x^{\alpha} \circ a$ which is a contradiction. Hence $x^{\alpha} \circ a=0$ for all $a \in P$.
 \end{proof}
Recall from \cite{ameri} that an element $x \in G$ is nilpotent if there exists an integer $t$ such that $0 \in x^t$. The set of all nilpotent elements of $G$ is denoted by $\Upsilon$. 
\begin{theorem}\label{11}
 Let $P$ be a weakly $(\alpha,\beta)$-prime $\mathcal{C}$-hyperideal of a strongly distributive multiplicative hyperring $H$ and $(x,y)$ be an $(\alpha,\beta)$-zero of $P$ where $\alpha,\beta\in \mathbb{Z}^+$. Then 
\begin{itemize}
\item[\rm(i)]~ If the hyperideal zero of $H$ is a strong $\mathcal{C}$-hyperideal, then $x\circ a \subseteq  \Upsilon$ for all $a \in P$.
\item[\rm(ii)]~ If the hyperideal zero of $H$ is a $\mathcal{C}$-hyperideal, $y \circ a \subseteq  \Upsilon $ for all $a \in P$.
 \end{itemize}
 \end{theorem}
 \begin{proof}
(i) Since $x^{\alpha} \circ a=0$ for all $a \in P$, by Proposition \ref{10} (3), we get the result that $x \circ a \subseteq  \Upsilon $.

(ii) Take any $a \in P$. Since $0 \in (x+a)^{\alpha} \circ y$ by Proposition \ref{10}(1) and the hyperideal zero of $H$ is a $\mathcal{C}$-hyperideal, we get $ ((x+a) \circ y)^{\alpha}=0$. This means that $(x+a) \circ y \subseteq   \Upsilon $. Also, since $0 \in x^{\alpha} \circ y$ and the hyperideal zero of $H$ is a $\mathcal{C}$-hyperideal, we have $ (x \circ y)^{\alpha}=0$ which implies $x \circ y \subseteq \Upsilon $. Since $H$ is a strongly distributive multiplicative hyperring, we have  $a \circ y=  (x+a) \circ y -x \circ y  \subseteq  \Upsilon $, as needed.
\end{proof}
 \begin{theorem}
 Let $\{P_i\}_{i \in I}$ be a family of weakly $(\alpha,\beta)$-prime hyperideals of $H$ and  $D(P_i)=\{x \in H \ \vert \ x^{\beta} \subseteq P_i\}$ for all $i \in I$ where $\alpha,\beta\in \mathbb{Z}^+$. If $D(P_i)=D(P_j)$ for all $i,j \in I$, then $\bigcap_{i \in I} P_i$ is a weakly $(\alpha,\beta)$-prime hyperideal of $H$.
 \end{theorem}
 \begin{proof}
 Assume that $0 \notin x^{\alpha} \circ y \subseteq \bigcap_{i \in I} P_i$ for $x,y \in H$ but $y \notin  \bigcap_{i \in I} P_i$. Therefore we conclude that  $y \notin P_j$ for some $j \in I$. Since $P_j$ is a weakly $(\alpha,\beta)$-prime hyperideal of $H$ and $0 \notin x^{\alpha} \circ y\subseteq P_j$, we get the result that $  x^{\beta} \subseteq  P_j$. This implies that $x \in D(P_j)$ and so $x \in D(P_i)$ for all $i \in I$ by the hypothesis. Then $x ^{\beta} \subseteq \bigcap_{i \in I} P_i$. This shows that $\bigcap_{i \in I} P_i$ is a a weakly $(\alpha,\beta)$-prime hyperideal of $H$.
 \end{proof}
Let $I$ be  a finite sum of finite products of elements of $H$. Consider the relation $\gamma$ on a multiplicative hyperring $H$ defined as $x \gamma y$ if and only if $\{x,y\} \subseteq I$, namely,  
$x \gamma y $ if and only if    $\{x,y\} \subseteq \sum_{j \in J} \prod_{i \in I_j} z_i$ for some $ z_1, ... , z_n \in H$ and $ I_j, J \subseteq \{1,... , n\}$.  $\gamma^{\ast}$ denotes the transitive closure of $\gamma$. The relation $\gamma^{\ast}$ is the smallest equivalence relation on $H$ such that the set of all equivalence classes, i.e., the
quotient $G/\gamma^{\ast}$,  is a fundamental ring. Assume that $\Sigma$
is the set of all finite sums of products of elements of $H$.  We can rewrite the
definition of $\gamma ^{\ast}$  on $H$, namely, 
$x\gamma^{\ast}y$ if and only if  there exist  $ z_1, ... , z_n \in H$ such that $z_1 = x, z_{n+1 }= y$ and $u_1, ... , u_n \in \Sigma$ where
$\{z_i, z_{i+1}\} \subseteq u_i$ for $1 \leq i \leq n$.
Suppose that $\gamma^{\ast}(x)$ is the equivalence class containing $x \in H$. Define $\gamma ^{\ast}(x) \oplus \gamma^{\ast}(y)=\gamma ^{\ast}(z)$ for every  $z \in \gamma^{\ast}(x) + \gamma ^{\ast}(y)$ and $\gamma ^{\ast}(x) \odot \gamma ^{\ast}(y)=\gamma ^{\ast}(w)$ for every  $w \in  \gamma^{\ast}(x) \circ \gamma ^{\ast}(y)$. Then $(H/\gamma ^{\ast},\oplus,\odot)$ is a ring called a fundamental ring of $H$ \cite{sorc4}.
\begin{theorem}\label{7}
Assume that $P$ is a hyperideal of $H$. Then  $P$  is a weakly  $(\alpha,\beta)$-prime hyperideal of $(H,+,\circ)$ if and only if $P/\gamma ^{\ast}$ is  a weakly  $(\alpha,\beta)$-prime ideal of $(H/\gamma ^{\ast},\oplus,\odot)$. 
\end{theorem}
\begin{proof}
 Let $0 \neq \underbrace{ x \odot \cdots \odot x}_{\alpha} \odot y \in P/\gamma ^{\ast}$  for some $x,y \in H/\gamma ^{\ast}$. Hence we have   $x =\gamma^{\ast}(a)$ and $y =\gamma^{\ast}(b)$ for some $a,b \in H$. This means  that $\underbrace{x\odot \cdots \odot x}_{\alpha} \odot y= \underbrace{\gamma^{\ast}(a) \odot \cdots \odot \gamma^{\ast}(a)}_{\alpha}  \odot \gamma^{\ast}(b) =\gamma^{\ast}(a^{\alpha} \circ b)$. Since $\gamma^{\ast}(0) \neq 
\gamma^{\ast}(a^{\alpha} \circ b) \in P/\gamma^{\ast}$, we get $0 \notin a^{\alpha} \circ b \subseteq  P$. Since $P$ is a weakly  $(\alpha,\beta)$-prime hyperideal of $H$, we get the result that  $a^{\beta} \subseteq P$ or $b \in P$. This implies that  $\underbrace{x \odot \cdots \odot x}_{\beta}=\underbrace{ \gamma^{\ast}(a) \odot \cdots \odot \gamma^{\ast}(a)}_{\beta}=\gamma^{\ast}(a^{\beta})\in P/\gamma ^{\ast}$ or $y=\gamma^{\ast}(a) \in P/\gamma ^{\ast}$. Consequently,   $P/\gamma ^{\ast}$ is a weakly  $(\alpha,\beta)$-prime  ideal of $H/\gamma ^{\ast}$.
\end{proof}
Let $(H_1,+_1,\circ_1)$ and $(H_2,+_2,\circ_2)$ be two multiplicative hyperrings with nonzero identity.  The set $H_1 \times H_2$  with the operation $+$ and the hyperoperation $\circ$  defined  as

$(x_1,x_2)+(y_1,y_2)=(x_1+_1y_1,x_2+_2y_2)$

$(x_1,x_2) \circ (y_1,y_2)=\{(x,y) \in H_1 \times H_2 \ \vert \ x \in x_1 \circ_1 y_1, y \in x_2 \circ_2 y_2\}$ \\
is a multiplicative hyperring \cite{ul}. Now, we present some characterizations of weakly $(\alpha,\beta)$-prime hyperideals on cartesian product of commutative multiplicative hyperring.

\begin{proposition} \label{malek}
Let $(H_1, +_1,\circ _1)$ and $(H_2,+_2,\circ_2)$ be two multiplicative hyperrings with  scalar identities $1_{H_1}$ and $1_{H_2}$,  respectively,  $P$ a proper nonzero hyperideal of   $H_1 \times H_2$, and $\alpha,\beta \in \mathbb{Z}^+$. If  $P$ is weakly $(\alpha,\beta)$-prime, then it has one of the following cases:
 \begin{itemize}
\item[\rm(i)]~ $P=P_1 \times H_2$ such that $P_1$ is an $(\alpha,\beta)$-prime hyperideal of $H_1 $.
\item[\rm(ii)]~ $P=H_1 \times P_2$ such that $P_2$ is an  $(\alpha,\beta)$-prime hyperideal of $H_2$.
 \end{itemize}
\end{proposition}
\begin{proof}
$\Longrightarrow$ Suppose that $P=P_1 \times P_2$ is a nonzero weakly $(\alpha,\beta)$-prime hyperideal of $H_1 \times H_2$ such that $P_1$ and $P_2$ are hyperideals of $H_1$ and $H_2$, respectively. Let us assume $P_1$ and $P_2$ are proper, and $P_1 \neq 0$. Take any $0 \neq x \in P_1$. Therefore we have $(0,0) \notin (1_{H_1},0)^{\alpha} \circ (x,1_{H_2}) \subseteq P_1 \times P_2$. Since $P$ is a nonzero weakly $(\alpha,\beta)$-prime hyperideal of $H_1 \times H_2$, we get the result that $(1_{H_1},0)^{\beta} \subseteq P_1 \times P_2$ or $(x,1_{H_2}) \in P_1 \times P_2$. It follows that $1_{H_1} \in P_1$ or $1_{H_2} \in P_2$. Then $P_1=H_1$ or $P_2=H_2$. This is a contradiction. Let us consider $P_1$ is proper and $P_2=H_2$. Now, we shows that $P_1$ is an $(\alpha,\beta)$-prime hyperideal of $H_1$. Let $x^{\alpha} \circ_1 y \subseteq P_1$ for $x,y \in H_1$. Hence we get $(0,0) \notin (x,1_{H_2})^{\alpha} \circ (y,1_{H_2}) \subseteq P_1 \times H_2$ and so we have $(x,1_{H_2})^{\beta} \subseteq P_1 \times H_2$ or $(y,1_{H_2}) \in P_1 \times H_2$. This implies that $x^{\beta} \subseteq P_1$ or $y \in P_1$. Similarly, it can be seen that if $P_1=H_1$ and $P_2$ is a proper hyperideal of $H_2$, then $P_2$ is $(\alpha,\beta)$-prime. 
\end{proof}
\begin{theorem}
Assume that  $H=H_1  \times \cdots \times H_n$ where $H_1,  \cdots, H_n$ are commutative multiplicative hyperrings,  $P$  a proper nonzero hyperideal of $H$ and $\alpha,\beta \in \mathbb{Z}^+$. Then the following are equivalent.
\begin{itemize}
\item[\rm(i)]~ $P$ is a weakly $(\alpha,\beta)$-prime hyperideal of $H$.
\item[\rm(ii)]~ $P=H_1 \times \cdots \times P_i \times \cdots \times H_n$ such that $P_i$ is an $(\alpha,\beta)$-prime hyperideal of $H_i$ for some $i \in \{1,\cdots,n\}$.
\item[\rm(iii)]~ $P$ is an $(\alpha,\beta)$-prime hyperideal of $H$.
\end{itemize}
\end{theorem}
\begin{proof}
(i) $\Longrightarrow$ (ii) Let $P=P_1 \times \cdots \times P_n$ is a weakly $(\alpha,\beta)$-prime hyperideal of $H$. We use the induction on $n$. If $n=2$, then the claim is true by Proposition \ref{malek}. Let the claim be true for $n-1$.
Assume that $I=P_1 \times \cdots \times P_{n-1}$. So $P=I \times P_n$. Then we conclude that  $I$ is an $(\alpha,\beta)$-prime hyperideal of $H_1 \times \cdots \times H_{n-1}$ and $P_n=H_n$ or $I=H_1 \times \cdots \times H_{n-1}$ and $P_n$  is an $(\alpha,\beta)$-prime hyperideal of $H_n$ by Proposition \ref{malek}. In the first possibility,  we obtain $I=H_1 \times \cdots \times P_i \times \cdots \times H_{n-1}$ such that $P_i$ is an $(\alpha,\beta)$-prime hyperideal of $H_i$ by induction hypothesis  and $P_n=H_n$. This shows that $P=H_1 \times \cdots \times P_i \times \cdots \times H_{n-1} \times H_n$ such that  $P_i$ is  an $(\alpha,\beta)$-prime hyperideal of $H_i$. In the second possibility, we have $I_i = H_i$  for all $i \in \{1,\cdots,n-1\}$ and $P_n$  is an $(\alpha,\beta)$-prime hyperideal of $H_n$. It follows that $P=H_1 \times \cdots \times H_{n-1} \times P_n$ where $P_n$ is is an $(\alpha,\beta)$-prime hyperideal of $H_n$.

(ii) $\Longrightarrow$ (iii) Without loss of generality, we assume that $P_1$ is an $(\alpha,\beta)$-prime hyperideal of $H_1$ and $P_i=H_i$ for all $i \neq 1$. Let $(x_1,x_2,\cdots,x_n)^{\alpha} \circ (y_1,y_2,\cdots,y_n) \subseteq P_1 \times H_2 \times \cdots \times H_n$ such that $(y_1,y_2,\cdots,y_n) \notin P_1 \times H_2 \times \cdots \times H_n$. This implies that $x_1^{\alpha} \circ_1 y_1 \subseteq P_1$ and $y_1 \notin P_1$. Since $P_1$ is an $(\alpha,\beta)$-prime hyperideal of $H_1$, we get $x_1^{\beta} \subseteq P_1$. It follows that $(x_1,x_2,\cdots,x_n)^{\beta} \subseteq P_1 \times H_2 \times \cdots \times H_n$ and this completes the proof.

(iii) $\Longrightarrow$ (i) It is straightforward.
\end{proof}


\end{document}